\documentclass[10pt]{amsart}
\usepackage{enumerate,amsmath,amssymb,latexsym,
amsfonts, amsthm, amscd, MnSymbol}

\theoremstyle{plain}

\newtheorem{theorem}{Theorem}
\newtheorem*{theorem*}{Theorem}

\numberwithin{equation}{section}

\newcommand{\ra}{\rightarrow}

\newcommand{\OO}{\Omega} 
\newcommand{\oo}{\omega} 
\newcommand{\ind}{{\rm ind}} 
\newcommand{\oD}{\overline{D}}

\addtocounter{section}{-1}

\begin{document}

\title {Fredholm theory for cofinite sets}

\date{}

\author[P.L. Robinson]{P.L. Robinson}

\address{Department of Mathematics \\ University of Florida \\ Gainesville FL 32611  USA }

\email[]{paulr@ufl.edu}

 \keywords{}

\begin{abstract}

We investigate two ways in which self-maps of an infinite set may be close to bijections; our investigation generates a $\mathbb{Z}$-valued index theory and a corresponding extension by $\mathbb{Z}$ for the quotient of the full symmetric group by its finitary subgroup. 

\end{abstract}

\maketitle

\medbreak

\section{Introduction} 

Let $\OO$ be an infinite set. Of the several ways in which we may regard a function $f: \OO \ra \OO$ as being close to bijective, we shall focus on two. According to the one, $f$ differs from a bijection on a finite subset of $\OO$; we call such an $f$ {\it almost bijective}. According to the other, $f$ restricts to a bijection between two cofinite subsets of $\OO$; we call such an $f$ a {\it near-bijection}.  

\medbreak 

Associated to any map $f: \OO \ra \OO$ are its {\it range}  $f(\OO) := \{ f(\oo) : \oo \in \OO \}$ and its {\it monoset} 
$$\OO_f :=  \{ \oo \in \OO : \overleftarrow{f} (\{f(\oo)\}) = \{ \oo \} \}$$ 
where $\overleftarrow{f}$ is the map that sends each subset of $\OO$ to its preimage under $f$. We shall write the complement of $A \subseteq \OO$ as $A' = \OO \setminus A$ and the cardinality of $A$ as $|A|$ when finite. In these terms, $f$ is a near-bijection iff both $f(\OO)'$ and $\OO_f'$ are finite. When $f$ is a near-bijection, we define its {\it index} by 
$$\ind (f) = (|\OO_f'| - |f(\OO_f')| ) - |f(\OO)'|.$$

\medbreak 

\begin{theorem*}
For a function $f : \OO \ra \OO$ the following conditions are equivalent: \par 
$\bullet$ $f$ is a near-bijection and $\ind (f) = 0;$ \par 
$\bullet$ $f$ is almost a bijection. 
\end{theorem*} 

We prove rather more than this. Let us say that two self-maps of $\OO$ are {\it almost equal} iff they differ on a finite set; of course, almost equality is an equivalence relation, which we shall denote by $\equiv$. Composition of near-bijections is compatible with almost equality: in fact, the $\equiv$-classes of near-bijections constitute a group $\mathbb{G}_{\OO}$; this contains the group $\mathbb{S}_{\OO}$ comprising all $\equiv$-classes of almost bijective functions. The index function is constant on $\equiv$-classes and indeed passes to a surjective group homomorphism ${\rm Ind}: \mathbb{G}_{\OO} \ra \mathbb{Z}$ with $\mathbb{S}_{\OO}$ as its kernel; thus we have a (split) short exact sequence 
$${\rm I} \ra \mathbb{S}_{\OO} \ra \mathbb{G}_{\OO} \ra \mathbb{Z} \ra 0.$$ 

\medbreak 

Along with the necessary supporting material, we prove some related results: for example, if $f$ is injective and $f(\OO)'$ is nonempty but finite, then $f$ is not almost equal to a surjection; in a similar vein, if $f$ is surjective and $\OO_f'$ is nonempty but finite, then $f$ is not almost equal to an injection. 

\medbreak 

As will be clear from this brief summary, much of the development echoes the standard Fredholm theory for operators, hence our title.  

\medbreak 

\section{Almost equality and near-bijections} 

\medbreak 

Let $\OO$ be an infinite set. Our concern throughout will be with self-maps of $\OO$: that is, with functions from $\OO$ to itself. 

\medbreak 

{\bf Definition:} Two self-maps of $\OO$ are {\it almost equal} iff they agree on a cofinite subset of $\OO$. 

\medbreak 

To the functions $f: \OO \ra \OO$ and $g: \OO \ra \OO$ we associate their {\it disagreement set} 
$$D(f, g) = \{ \oo : f(\oo) \neq g(\oo) \}.$$
In these terms, $f$ and $g$ are almost equal iff $D(f, g)$ is finite. To indicate that $f$ and $g$ are almost equal, we shall write $f \equiv g$. 

\begin{theorem} 
Almost equality is an equivalence relation.
\end{theorem} 

\begin{proof} 
Only transitivity requires any effort; if $f, g, h$ are self-maps of $\OO$ then 
$$D(f, h) \subseteq D(f, g) \cup D(g, h).$$
\end{proof} 

\medbreak 

As usual, we shall indicate the $\equiv$-class of a function $f$ by $[f]$. 

\medbreak 

One way in which a self-map of $\OO$ can be close to a bijection is for it to be almost equal to a bijection; we say that such a map is {\it almost bijective}. Thus, an almost bijective map arises precisely by changing the values of a bijection on a finite set. A less restrictive way in which a self-map of $\OO$ can be close to a bijection is as follows. 

\medbreak 

{\bf Definition:} The self-map $f: \OO \ra \OO$ is a {\it near-bijection} (or is {\it nearly bijective}) iff there exist cofinite subsets $A \subseteq \OO$ and $B \subseteq \OO$ such that the restriction $A \xrightarrow{f} B$ is a bijection. 

\medbreak 

Of course, bijections are certainly near-bijections. More generally, self-maps that are almost equal to bijections are near-bijections. 

\begin{theorem} \label{almostnear}
If $f: \OO \ra \OO$ is almost bijective then $f$ is a near-bijection. 
\end{theorem} 

\begin{proof} 
Let $f \equiv g$ with $g$ bijective. Say $f$ and $g$ agree on the cofinite set $A$ and let $B = f(A) = g(A)$. As $g$ is bijective, $\OO \setminus B = g(\OO \setminus A)$ is finite. By construction, $f$ maps $A$ to $B$ bijectively. 
\end{proof} 

\medbreak 

In the opposite direction, we may ask whether or not it is possible to change the values of a given near-bijection on a finite subset so as to produce a true bijection. We shall answer this question completely and precisely in Theorem \ref{indzero}. 

\medbreak 

Now let $f$ be any self-map of $\OO$. So as not to suggest by infelicitous notation that a map has an inverse, we shall write $\overleftarrow{f}$ for the map that assigns to each subset of $\OO$ its preimage under $f$: thus, if $B \subseteq \OO$ then 
$$\overleftarrow{f} (B) = \{ \oo \in \OO : f(\oo) \in B \}.$$ 
The {\it range} of $f$ is as usual the subset $f(\OO)$ of $\OO$ defined by 
$$f(\OO) = \{ f(\oo) : \oo \in \OO \}.$$
The {\it monoset} of $f$ is the subset $\OO_f$ of $\OO$ defined by 
$$\OO_f = \{ \oo \in \OO : \overleftarrow{f} (\{ f(\oo) \}) = \{ \oo \} \}.$$
Note that $f$ is: surjective iff $f(\OO)' = \emptyset$; injective iff $\OO_f' = \emptyset$. 

\medbreak 

The following elementary sufficient conditions for near-bijectivity will soon be superceded. 

\begin{theorem} \label{insur}
The self-map $f$ of $\OO$ is a near-bijection if: either (i) $f$ is injective and $f(\OO)$ is cofinite; or (ii) $f$ is surjective and $\OO_f$ is cofinite. 
\end{theorem} 

\begin{proof} 
(i) is clear: $f$ gives a bijection from $\OO$ to $f(\OO)$. For (ii) let $A = \OO_f$ and $B = f(\OO_f)$: of course, $A \xrightarrow{f} B$ is a bijection; by hypothesis $A$ is cofinite, while $\OO \setminus B = f(\OO) \setminus f(A) \subseteq f(\OO \setminus  A)$ is finite. 
\end{proof} 

\medbreak 

{\bf Remark:} In the proof, $\OO \setminus B = f(\OO \setminus A)$: indeed, if $\oo' \in f(\OO \setminus A)$ then $\oo' = f(\oo)$ for some $\oo \in \OO \setminus A$; if also $\oo' = f(\oo_0)$ for some $\oo_0 \in A$ then $\overleftarrow{f} (\{ f(\oo_0) \}) \supseteq \{ \oo_0, \oo \}$ (contradiction) so $\oo' \notin f(A)$. On this point, see also Theorem \ref{comp}. 

\medbreak 

As we noted in Theorem \ref{almostnear}, a function that is almost bijective is a near-bijection. The converse of this statement is false. 

\begin{theorem} \label{in}
Let $f: \OO \ra \OO$ be injective and let $f(\OO)$ be cofinite and proper. If $g: \OO \ra \OO$ is surjective then $g \nequiv f$. 
\end{theorem} 

\begin{proof} 
We show that the disagreement set $D := D(f, g)$ is infinite. Fix a choice $\overline{\oo} \in f(\OO)'$; thus, $f$ does not have $\overline{\oo}$ as a value. Choose $\oo_0 \in \OO$ such that $g(\oo_0) = \overline{\oo}$; this is possible, as $g$ is surjective. Of course, $f$ and $g$ disagree at $\oo_0$; that is, $\oo_0 \in D$. Choose $\oo_1 \in \OO$ such that $g(\oo_1) = f(\oo_0)$. \par
Claim: $\oo_1 \neq \oo_0$ and $\oo_1 \in D$. 
[If $\oo_0 = \oo_1$ then $g(\oo_0) = g(\oo_1) = f(\oo_0)$ (contradicting $\oo_0 \in D$); thus $\oo_0 \neq \oo_1$. If $\oo_1 \notin D$ then $f(\oo_1) = g(\oo_1) = f(\oo_0)$ so (as $f$ is injective) $\oo_1 = \oo_0$ (contradiction); thus $\oo_1 \in D$.] \par 
This may be repeated inductively to produce a sequence $(\oo_n)_{n = 0}^{\infty}$ of distinct points in $D$. 
\end{proof} 

\medbreak 

Thus, such near-bijections (see Theorem \ref{insur}(i)) are not even almost equal to surjections. 

\medbreak 

\begin{theorem} \label{sur}
Let $f: \OO \ra \OO$ be surjective and let $\OO_f$ be cofinite and proper. If $g: \OO \ra \OO$ is injective then $g \nequiv f$. 
\end{theorem} 

\begin{proof} 
Again, write $D$ in place of $D(f, g)$. By hypothesis, $f$ takes the same value at certain points of $\OO$ but $g$ separates them; accordingly, we may choose $\oo_0 \in \OO_f'$ such that $f(\oo_0) \neq g(\oo_0)$ whereupon $\oo_0 \in D$. As $f$ is surjective, we may choose $\oo_1 \in \OO$ such that $f(\oo_1) = g(\oo_0)$. \par
Claim: $\oo_1 \neq \oo_0$ and $\oo_1 \in D$. [If $\oo_0 = \oo_1$ then $f(\oo_0) = f(\oo_1) = g(\oo_0)$ (contradicting $\oo_0 \in D$); thus $\oo_0 \neq \oo_1$. If $\oo_1 \notin D$ then $g(\oo_1) = f(\oo_1) = g(\oo_0)$ so (as $g$ is injective) $\oo_1 = \oo_0$ (contradiction); thus $\oo_1 \in D$.] \par 
Again, the process of induction manufactures a sequence $(\oo_n)_{n = 0}^{\infty}$ of distinct points in $D$. 
\end{proof} 

\medbreak 

Thus, such near-bijections (see Theorem \ref{insur}(ii)) are not even almost equal to injections. 

\medbreak 

These results prompt the following. 

\medbreak 

{\bf Definition:} The self-map $f$ of $\OO$ is: a {\it near-surjection} (or is {\it nearly surjective}) iff $f(\OO)$ is cofinite; a {\it near-injection} (or is {\it nearly injective}) iff $\OO_f$ is cofinite. 

\medbreak 

Theorem \ref{in} says that if a near-surjective injection is not surjective, then it is not almost equal to a surjective function; Theorem \ref{sur} says that if a near-injective surjection is not injective, then it is not almost equal to an injective function. 

\medbreak 

There is another natural candidate for near-injectivity. 

\medbreak 

\begin{theorem} \label{injeq}
The self-map $f$ of $\OO$ is a near-injection iff $\OO$ has a cofinite subset on which $f$ is injective. 
\end{theorem} 

\begin{proof} 
($\Rightarrow$) By definition, $f$ is injective on the cofinite set $\OO_f$. \\ 
($\Leftarrow$) Suppose $f$ to be injective on the cofinite subset $A \subseteq \OO$. Of course, the complement $\OO_f' = \OO \setminus \OO_f$ meets $A' = \OO \setminus A$ in a finite set. Now consider the restriction $A \cap \OO_f' \xrightarrow{f} \OO$. By supposition, this restriction is certainly injective. Let $\oo \in A \cap \OO_f'$: there exists $\oo' \neq \oo$ such that $f(\oo') = f(\oo)$ (as $\oo \in \OO_f'$) and $\oo' \notin A$ (as $f|_A$ is injective); it follows that $f(\oo) = f(\oo') \in f(\OO \setminus A)$. Thus $f$ is injective on $A \cap \OO_f'$ and there takes its values in the finite set $f(\OO \setminus A)$; so $A \cap \OO_f'$ is finite. Finally, $\OO_f' = (A' \cap \OO_f') \cup (A \cap \OO_f')$ is finite. 
\end{proof} 

\medbreak 

The relationship between the `near' versions of bijection, injection and surjection is exactly as we should wish. 

\begin{theorem} \label{near}
For a self-map $f$ of $\OO$ the following conditions are equivalent: \par
$\bullet$ $f$ is both nearly injective and nearly surjective; \par 
$\bullet$ $f$ is a near-bijection. 
\end{theorem} 

\begin{proof} 
($\Uparrow$) Let $f$ be nearly bijective: say $f$ maps the cofinite set $A \subseteq \OO$ to the cofinite set $B \subseteq \OO$ bijectively. As $f|_A$ is injective, $f$ is a near-injection by Theorem \ref{injeq}; $f$ is a near-surjection because $f(\OO) \supseteq B$. \par
($\Downarrow$) Let $\OO_f$ and $f(\OO)$ be cofinite. Consider $f$ as a map from $\OO_f$ to $f(\OO_f)$. This map is both injective (by definition of $\OO_f$) and surjective (by definition of $f(\OO_f)$). As given, $\OO \setminus \OO_f$ is finite; $\OO \setminus f(\OO_f)$ is finite, because 
$$\OO \setminus f(\OO_f) \subseteq (\OO \setminus f(\OO)) \cup (f(\OO) \setminus f(\OO_f)) \subseteq (\OO \setminus f(\OO)) \cup f(\OO \setminus \OO_f). $$
\end{proof} 

\medbreak 

We now wish to explore the interaction between composition and the `near' and `almost' notions that we have introduced.  

\medbreak 

The set of near-surjections is closed under composition. 

\medbreak 

\begin{theorem} 
If $f: \OO \ra \OO$ and $g: \OO \ra \OO$ are near-surjections, then so is $g \circ f$. 
\end{theorem} 

\begin{proof} 
Simply observe that
$$\OO \setminus g \circ f (\OO) = \OO \setminus g(f(\OO)) \subseteq (\OO \setminus g(\OO)) \cup (g(\OO) \setminus g(f(\OO)))$$
where $\OO \setminus g(\OO)$ is finite (as $g$ is a near-surjection) and $g(\OO) \setminus g(f(\OO)) \subseteq g(\OO \setminus f(\OO))$ is finite (as $f$ is a near-surjection). 
\end{proof} 

\medbreak 

The set of near-injections is closed under composition. 

\begin{theorem} \label{incom}
If $f: \OO \ra \OO$ and $g: \OO \ra \OO$ are near-injections, then so is $g \circ f$. 
\end{theorem} 

\begin{proof} 
By Theorem \ref{injeq} we may assume that $f$ is injective on the cofinite set $A$ and $g$ is injective on the cofinite set $B$. \par 
Claim: $g \circ f$ is injective on $A \cap \overleftarrow{f}(B)$. [Let $\oo, \oo' \in A \cap \overleftarrow{f}(B)$ satisfy $g \circ f (\oo) = g \circ f (\oo')$ so that $g(f(\oo)) = g(f(\oo'))$: as $f(\oo), f(\oo') \in B$ and $g|_B$ is injective, $f(\oo) = f(\oo')$; as $\oo, \oo' \in A$ and $f|_A$ is injective, $\oo = \oo'$.] \par 
Claim: $A \cap \overleftarrow{f}(B)$ is cofinite. [As $A'$ is finite, we need only check that $\overleftarrow{f} (B)' = \overleftarrow {f} (B')$ is finite; again, as $A'$ is finite, we need only check that $A \cap \overleftarrow{f} (B')$ is finite. This is clear: $A \cap \overleftarrow{f} (B') = \overleftarrow{f|_A} (B')$ with $f|_A$ injective and $B'$ finite.] 
\end{proof} 

\medbreak 

Of course, the next result is now immediate. 

\begin{theorem} \label{bicom}
If $f: \OO \ra \OO$ and $g: \OO \ra \OO$ are near-bijections, then so is $g \circ f$. 
\end{theorem} 

\begin{proof}
Simply combine the last two results. 
\end{proof} 

\medbreak 

So much for the interaction between composition and our `near' notions; now for the interaction between composition and almost equality. 

\medbreak 

Let $f_1, f_2, g_1, g_2$ be self-maps of $\OO$ such that $f_1 \equiv f_2$ and $g_1 \equiv g_2$. We wish to know whether it neccessarily follows that $g_1 \circ f_1 \equiv g_2 \circ f_2$. For convenience, let us agree to use the abbreviations $D_f = D(f_1, f_2)$, $D_g = D(g_1, g_2)$, $D_{gf} = D(g_1 \circ f_1, g_2 \circ f_2)$. With this agreement, we wish to know whether finiteness of $D_f \cup D_g$ forces that of $D_{gf}$. 

\medbreak 

With no additional assumptions, the answer to this question is negative. For example, choose $\oo_0 \in D_g$ and let $f_1|_{D_f'} = f_2|_{D_f'}$ take constant value $\oo_0$; then the disagreement set $D_{gf}$ contains the cofinite set $D_f'$. If we assume that $f_1$ and $f_2$ are near-injections, then all is well. 

\begin{theorem} \label{com}
Let $f_1, f_2, g_1, g_2$ be self-maps of $\OO$ such that $f_1 \equiv f_2$ and $g_1 \equiv g_2$. If $f_1$ and $f_2$ are near-injections, then $g_1 \circ f_1 \equiv g_2 \circ f_2$.
\end{theorem} 

\begin{proof} 
Abbreviate by $f: D_f' \ra \OO$ the coincident restrictions of $f_1$ and $f_2$ to the complement of their disagreement set. If $\oo \in D_f' \cap D_{gf}$ then 
$$g_1 (f(\oo)) = g_1 \circ f_1 (\oo) \neq g_2 \circ f_2 (\oo) = g_2 (f(\oo));$$
thus 
$$D_f' \cap D_{gf} \subseteq \overleftarrow{f} (D_g).$$
As $D_f$ is finite, finiteness of $D_{gf}$ is therefore an immediate consequence of the following claim. \par
Claim: $D_f' \cap \overleftarrow{f} (D_g)$ is finite. [Let $f_1$ and $f_2$ be injective on the cofinite sets $A_1$ and $A_2$; the intersection $A = A_1 \cap A_2$ is cofinite, of course. Now 
$$D_f' \cap \overleftarrow{f} (D_g) = \overleftarrow{f|_{D_f' \cap A}} (D_g) \cup \overleftarrow{f|_{D_f' \setminus A}} (D_g)$$
where the first set on the right is finite (because $f|_{D_f' \cap A}$ is injective and $D_g$ finite) and the second set on the right is finite (because it is contained in $D_f' \setminus A = D_f' \cap A' \subseteq A'$ and $A'$ is finite).]
\end{proof} 

\medbreak 

\section{Index theory} 

\medbreak 

In this section we develop an index theory for near-bijections, associating to each near-bijection $f$ an integer ${\rm ind} (f)$ that vanishes precisely when $f$ is almost bijective. 

\medbreak 

The following observation regarding any self-map of $\OO$ will be useful. It formalizes the remark following Theorem \ref{insur}. 

\begin{theorem} \label{comp}
If $f$ is any self-map of $\OO$ then
$$f(\OO \setminus \OO_f) = f(\OO) \setminus f(\OO_f);$$
in particular, if $f$ is surjective then 
$$f(\OO_f') = f(\OO_f)'.$$
\end{theorem}

\begin{proof} 
($\subseteq$) Let $\oo' \in \OO \setminus \OO_f$: if $f(\oo') \in f(\OO_f)$ then $f(\oo') = f(\oo)$ for some $\oo \in \OO_f$ so that $\oo' = \oo$ (by definition of $\OO_f$) and therefore $\oo' \in \OO_f$ (contradiction); so $f(\oo') \in f(\OO) \setminus f(\OO_f)$. \par
($\supseteq$) Let $\overline{\oo} \in f(\OO) \setminus f(\OO_f)$ and indeed let $\overline{\oo} = f(\oo)$: if $\oo \in \OO_f$ then $\overline{\oo} = f(\oo) \in f(\OO_f)$ (contradiction); thus $\oo \in \OO \setminus \OO_f$ and so $\overline{\oo} = f(\oo) \in f(\OO \setminus \OO_f)$. 
\end{proof} 

\medbreak 

{\bf Remark:} It will also be useful to notice that if $f$ is any self-map of $\OO$ then $\OO_f'$ is the disjoint union 
$$\OO_f' = \bigcupdot_{\overline{\oo} \in f(\OO_f')} \overleftarrow{f} (\{ \overline{\oo} \}).$$

\medbreak 

For the present paragraph only, we suspend our standing convention and take $\OO$ to be {\it finite}. Let $f: \OO \ra \OO$ be any map. Theorem \ref{comp} gives us the disjoint union 
$$f(\OO) = f(\OO_f) \cupdot f(\OO_f')$$
so that 
$$|f(\OO)| = |f(\OO_f)| + |f(\OO_f')| = |\OO_f| + |f(\OO_f')|$$
as $f|_{\OO_f}$ is injective; consequently 
$$|f(\OO)'| = |\OO| - |f(\OO)| = |\OO| - |\OO_f| - |f(\OO_f')| = |\OO_f'| - |f(\OO_f')|$$
and in conclusion
$$|\OO_f'| - |f(\OO_f')| = |f(\OO)'|.$$

\medbreak 

Reinstate the standing convention that $\OO$ is infinite. We extend the conclusion of the last paragraph to self-maps that are almost equal to the identity, as a convenient step towards the handling of arbitrary almost bijective self-maps. 

\begin{theorem} \label{fin}
If $f: \OO \ra \OO$ satisfies $f \equiv {\rm I}$ then 
$$|\OO_f'| - |f(\OO_f')| = |f(\OO)'|.$$
\end{theorem} 

\begin{proof} 
Let $f$ and the identity map ${\rm I}$ have finite disagreement set $D$. Define $\oD := D \cup f(D)$. Let $\oo \in \oD$: if $\oo \in D$ then $f(\oo) \in f(D) \subseteq \oD$; if $\oo \notin D$ then $f(\oo) = \oo \in \oD$. Thus $f$ maps $\oD$ to itself. Let us write $g$ for $f$ as a map from $\oD$ to $\oD$. 

Now 
$$f(\OO) = f(\oD) \cup f(\OO \setminus \oD) = g(\oD) \cup (\OO \setminus \oD)$$
as $f$ fixes points of $\OO \setminus D \supseteq \OO \setminus \oD$; hence 
$$\OO \setminus f(\OO) = (\OO \setminus g(\oD)) \cap \oD = \oD \setminus g(\oD)$$
and in short 
$$\OO \setminus f(\OO) = \oD \setminus g(\oD).$$

We claim that  
$$\OO \setminus \OO_f = \oD \setminus \oD_g.$$ 
The inclusion $\supseteq$ is easy. Let $\oo \in \oD \setminus \oD_g$: by definition, there exists $\oo' \neq \oo$ in $\oD$ such that $g(\oo') = g(\oo)$; of course, as $g = f|_{\oD}$ it follows that $f(\oo') = f(\oo)$. Conclusion: $\oo \in \OO \setminus \OO_f$. The inclusion $\subseteq$ is a little less easy. Let $\oo \in \OO \setminus \OO_f$; by definition, there exists $\oo' \neq \oo$ in $\OO$ such that $f(\oo') = f(\oo)$. Consider the placement of $\oo$ and $\oo'$. If both were to lie in $\OO \setminus \oD$ then $f$ would fix both and render them equal; so this case does not arise. Suppose $\oo \in \oD$ and $\oo' \in \OO \setminus \oD$: then $f$ fixes $\oo'$ and therefore $\OO \setminus \oD \ni \oo' = f(\oo') = f(\oo) \in \oD$ (contradiction); so this case also does not arise. The case $\oo \in \OO \setminus \oD$ and $\oo' \in \oD$ is ruled out likewise. We deduce that $\oD$ contains both $\oo$ and $\oo'$ and conclude that $\oo \in \oD \setminus \oD_g$ by definition. 

As an immediate consequence, also 
$$f(\OO \setminus \OO_f) = g(\oD \setminus \oD_g).$$

From the finite case handled just prior to the theorem, 
$$|\oD \setminus \oD_g| - |g(\oD \setminus \oD_g)| = |\oD \setminus g(\oD)|$$
whence 
$$|\OO \setminus \OO_f| - |f(\OO \setminus \OO_f)| = |\OO \setminus f(\OO)|.$$
\end{proof} 

\medbreak 

In order to extend this result to arbitrary almost bijective self-maps, we analyze the effect on range and monoset of composition with permutations (that is, true self-bijections). As is customary, we shall write $S_{\OO}$ for the symmetric group comprising all permutations of $\OO$. Only one of the next pair of theorems will be used at once; both will come into play later. 

\medbreak 

On the one side the effect is as follows. 

\begin{theorem} \label{left}
If $\pi \in S_{\OO}$ and if $f$ is any self-map of $\OO$ then 
$$\pi \circ f (\OO) = \pi (f(\OO)), \; \; \;  \OO_{\pi \circ f} = \OO_f.$$ 
\end{theorem} 

\begin{proof} 
Write $g = \pi \circ f$ for short. The first equality holds clearly when $\pi$ is any self-map of $\OO$. Let $\oo \in \OO_f$: if $\oo' \in \overleftarrow{g} (\{g(\oo)\})$ then $\pi (f(\oo')) = g(\oo') = g(\oo) = \pi (f(\oo))$ so $f(\oo') = f(\oo)$ and therefore $\oo' = \oo$; this places $\oo$ in $\OO_g$. Let $\oo \in \OO_g$: if $\oo' \in \overleftarrow{f} (\{f(\oo)\})$ then $g(\oo') = \pi (f(\oo')) = \pi (f(\oo)) = g(\oo)$ and therefore $\oo' =\oo$; this places $\oo$ in $\OO_f$. 
\end{proof} 

\medbreak 

The effect on the other side is as follows. 

\medbreak 

\begin{theorem} \label{right}
If $\pi \in S_{\OO}$ and if $f$ is any self-map of $\OO$ then 
$$f \circ \pi (\OO) = f(\OO), \; \; \; \OO_{f \circ \pi} = \overleftarrow{\pi} (\OO_f).$$
\end{theorem} 

\begin{proof} 
Write $g = f \circ \pi$ for short. The first equality holds because $\pi$ is surjective. Let $\oo \in \overleftarrow{\pi} (\OO_f)$ so $\pi (\oo) \in \OO_f$: if $\oo' \in \overleftarrow{g} (\{ g(\oo) \})$ then $f( \pi(\oo')) = g(\oo') = g(\oo) = f( \pi(\oo))$ so $\pi(\oo') = \pi(\oo)$ and therefore $\oo' = \oo$; this places $\oo$ in $\OO_g$. Let $\oo \in \OO_g$: if $\pi(\oo') \in \overleftarrow{f} (\{f(\pi(\oo))\})$ then $g(\oo') = f(\pi(\oo')) = f(\pi(\oo)) = g(\oo)$ so $\oo' = \oo$ and therefore $\pi(\oo') = \pi(\oo)$; this places $\oo$ in $\overleftarrow{\pi} (\OO_f)$. 
\end{proof} 

\medbreak 

We are now positioned to extend Theorem \ref{fin} to arbitrary almost bijective self-maps. 

\begin{theorem} \label{zero} 
If $f$ is an almost bijective self-map of $\OO$ then 
$$|\OO_f'| - |f(\OO_f')| = |f(\OO)'|.$$
\end{theorem} 

\begin{proof} 
Let $g: \OO \ra \OO$ be a bijection such that $f \equiv g$; it follows that $f \circ g^{-1} \equiv {\rm I}$. Theorem \ref{fin} provides us with the equality 
$$|\OO_{f \circ g^{-1}}'| = |(f \circ g^{-1}) (\OO_{f \circ g^{-1}}')| + |(f \circ g^{-1}) (\OO)'|$$
 while Theorem \ref{right} and the taking of complements show that 
$$|\OO_{f \circ g^{-1}}'| = |g(\OO_f')| = |\OO_f'|, \; \; \; (f \circ g^{-1}) (\OO_{f \circ g^{-1}}') = f(\OO_f'), \; \; \; (f \circ g^{-1}) (\OO)' = f(\OO)'$$
and the proof is complete.
\end{proof} 

\medbreak 

{\bf Remark:} In the proof, we may instead argue from $g^{-1} \circ f \equiv {\rm I}$ and apply Theorem \ref{left}. 

\medbreak 

Before moving on, we pause to note one immediate consequence: the familiar fact that injectivity and surjectivity are equivalent for self-maps of a finite set continues to hold true for almost bijective self-maps in general. 

\medbreak 

\begin{theorem} 
For an almost bijective map $f: \OO \ra \OO$ the following conditions are equivalent: \par 
$\bullet$ $f$ is injective, \par 
$\bullet$ $f$ is surjective;\\
they say that $f$ is a bijection.
\end{theorem} 

\begin{proof} 
($\Downarrow$) If $f$ is injective then $\OO_f' = \emptyset$ and $f(\OO_f') = \emptyset$ so that $|f(\OO)'| = 0$ and $f$ is surjective. \par
($\Uparrow$) Let $f$ be surjective: then $f(\OO)' = \emptyset$ so that $|\OO_f'| = |f(\OO_f')|$; as the finite set $\OO_f'$ would shrink under $f$ if it were nonempty, we deduce that $\OO_f' = \emptyset$ and conclude that $f$ is injective. 
\end{proof} 

\medbreak 

We are prompted to define a notion of `index' for arbitrary near-bijections. Note that if the self-map $f$ of $\OO$ is a near-bijection, then $\OO_f'$ and $f(\OO)'$ are finite, so the following definition makes sense. 

\medbreak 

{\bf Definition:} The {\it index} of the near-bijection $f: \OO \ra \OO$ is the integer ${\rm ind} (f) \in \mathbb{Z}$ defined by 
$${\rm ind} (f) := (|\OO_f'| - |f(\OO_f')|) - |f(\OO)'|.$$

\medbreak 

The motivating Theorem \ref{zero} informs us that the index of an almost bijective self-map is zero; in fact the converse holds, so that vanishing of the index precisely characterizes almost bijectivity. 

\medbreak 

\begin{theorem} \label{indzero}
The near-bijection $f: \OO \ra \OO$ is almost bijective iff ${\rm ind} (f) = 0.$
\end{theorem} 

\begin{proof} 
Assume that ${\rm ind} (f) = 0$; that is, assume the equality 
$$|\OO_f'| - |f(\OO_f')| = |f(\OO)'|.$$
Recall from the Remark after Theorem \ref{comp} that 
$$\OO_f' = \bigcup_{\overline{\oo} \in f(\OO_f')} \overleftarrow{f} (\{ \overline{\oo} \}).$$
For each $\overline{\oo} \in f(\OO_f')$ remove all but one element from $\overleftarrow{f} (\{ \overline{\oo} \})$; the points so removed are precisely $|\OO_f'| - |f(\OO_f')| =  |f(\OO)'|$ in number. Redefine $f$ at each of these points so that the new values of $f$ there make up $f(\OO)'$. In the process, we `empty' $\OO_f'$ and `fill' $f(\OO)'$; the resulting function $g: \OO \ra \OO$ is a bijection that agrees with $f$ on $\OO_f$. 
\end{proof} 

\medbreak 

The index is insensitive to changes on a finite set. 

\begin{theorem} \label{insensitive}
Let $f: \OO \ra \OO$ and $g: \OO \ra \OO$ be near-bijections. If $f \equiv g$ then ${\rm ind} (f) = {\rm ind} (g)$.  
\end{theorem} 

\begin{proof} 
By induction on the cardinality of the difference set, we may assume that $f$ and $g$ differ at just one point: say $g(\oo_0) = \oo_1 \neq f(\oo_0)$ but $g(\oo) = f(\oo)$ whenever $\oo \neq \oo_0$. We consider separately the cases $\oo_1 \in f(\OO)'$ and $\oo_1 \in f(\OO)$. \par 

Suppose that $\oo_1 \in f(\OO)'$. Either {\bf (i)} $\oo_0 \in \OO_f$ or {\bf (ii)} $\oo_0 \in \OO_f'$. {\bf (i)} If $\oo_0 \in \OO_f$ then passage from $f$ to $g$ loses $f(\oo_0)$ as a value but gains $\oo_1$; thus $|g(\OO)'| = |f(\OO)'|$ and $\OO_g' = \OO_f'$ so the index is unchanged. {\bf (ii)} If $\oo_0 \in \OO_f'$ then $f(\oo_0)$ is still a value and $\oo_1$ is gained, so $|g(\OO)'| = |f(\OO)'| - 1$; at the same time, either $|\overleftarrow{f} (\{ f(\oo_0) \})| = 2$ (in which case $|\OO_f'|$ falls by $2$ and $|f(\OO_f')|$ falls by $1$) or $|\overleftarrow{f} (\{ f(\oo_0) \})| > 2$ (in which case $|\OO_f'|$ falls by $1$ and $|f(\OO_f')|$ is unchanged) so that $|\OO_g'| - |g(\OO_g')| = |\OO_f'| - |f(\OO_f')| - 1$. Again, the index is unchanged. \par 
Suppose that $\oo_1 \in f(\OO)$: say $f(\overline{\oo}) = \oo_1 \neq f(\oo_0)$; of course, $\overline{\oo} \neq \oo_0$. Again, either {\bf (iii)} $\oo_0 \in \OO_f$ or {\bf (iv)} $\oo_0 \in \OO_f'$. {\bf (iii)} If $\oo_0 \in \OO_f$ then $f(\oo_0)$ is lost as a value but $\oo_1$ is not gained, so $|f(\OO)'|$ is increased by one. At the same time, either $\overline{\oo} \in \OO_f$ (so that $\oo_0$ and $\overline{\oo}$ are added to $\OO_f'$ and $f(\OO_f')$ picks up one element) or $\overline{\oo} \in \OO_f'$ (so that $\OO_f'$ picks up one element and $f(\OO_f')$ is unchanged); in each case, $|\OO_f'| - |f(\OO_f')|$ is increased by one. Once again, the index is unchanged. {\bf (iv)} If $\oo_0 \in \OO_f'$ then $f(\oo_0)$ remains a value and $\oo_1$ is not gained, so $|f(\OO)'|$ is unchanged. We again consider the cases $|\overleftarrow{f} (\{ f(\oo_0) \})| = 2$ and $|\overleftarrow{f} (\{ f(\oo_0) \})| > 2$ separately. In case $|\overleftarrow{f} (\{ f(\oo_0) \})| = 2$, say $\overleftarrow{f} (\{ f(\oo_0) \}) = \{ \oo_0, \oo_0' \}$: either $\overline{\oo} \in \OO_f$ (when $\OO_f'$ loses $\oo_0'$ but gains $\overline{\oo}$ (and $\oo_0$ is retained but with switched value) while $f(\OO_f')$ loses $f(\oo_0)$ but gains $\oo_1$) or $\overline{\oo} \in \OO_f'$ (when $\OO_f'$ loses $\oo_0'$ while $f(\OO_f')$ loses $f(\oo_0') = f(\oo_0)$); either way, $|\OO_f'| - |f(\OO_f')|$ is unchanged. In case $|\overleftarrow{f} (\{ f(\oo_0) \})| > 2$, either $\overline{\oo} \in \OO_f$ (when $\OO_f'$ picks up $\overline{\oo}$ while $f(\OO_f')$ picks up $\oo_1$) or $\overline{\oo} \in \OO_f'$ (when $\OO_f'$ and $f(\OO_f')$ are unchanged); either way,  $|\OO_f'| - |f(\OO_f')|$ is unchanged. Finally, the index is unchanged once more. \par 
We have now considered all cases; the proof is complete. 
\end{proof} 

\medbreak 

{\bf Remark:} This provides an alternative route to the results of Theorem \ref{in} and Theorem \ref{sur} when the map $g$ considered there is a near-bijection: in Theorem \ref{in}, ${\rm ind} (f) < 0 \leqslant {\rm ind} (g)$; in Theorem \ref{sur}, ${\rm ind} (f) > 0 \geqslant {\rm ind} (g).$

\medbreak 

The index is also invariant under postcomposition and precomposition by permutations. 

\begin{theorem} \label{coset}
Let $f: \OO \ra \OO$ be a near-bijection. If $\pi \in S_{\OO}$ is any permutation then 
$${\rm ind} (\pi \circ f) = {\rm ind} (f) = {\rm ind} (f \circ \pi).$$
\end{theorem} 

\begin{proof} 
Theorem \ref{left} yields ${\rm ind} (\pi \circ f) = {\rm ind} (f) $ while Theorem \ref{right} yields $ {\rm ind} (f) = {\rm ind} (f \circ \pi).$
\end{proof} 

\medbreak 

In the opposite direction, we claim that if the near-bijections $f$ and $g$ of $\OO$ have equal index, then there exist permutations $\lambda$ and $\rho$ such that $\lambda \circ f \equiv g \equiv f \circ \rho$. After a simplifying reduction, our justification of this claim will come in four parts, exhibiting $\lambda$ and $\rho$ when the index is negative and when the index is positive. 

\medbreak 

The aforementioned simplification is as follows. 

\medbreak 

\begin{theorem} \label{red} 
Let $f: \OO \ra \OO$ be a near-bijection. If ${\rm ind} (f) \leqslant 0$ then there exists an injection $g: \OO \ra \OO$ such that $g \equiv f$ and $|g (\OO)'| = - {\rm ind} (f)$. If ${\rm ind} (f) \geqslant 0$ then there exists a surjection $g: \OO \ra \OO$ such that $g \equiv f$ and $|\OO_g'| - |f(\OO_g')| = {\rm ind} (f).$
\end{theorem} 

\begin{proof} 
Once again, recall the disjoint decomposition 
$$\OO_f' = \bigcupdot_{\overline{\oo} \in f(\OO_f')} \overleftarrow{f} (\{ \overline{\oo} \}).$$
In case $|\OO_f'| - |f(\OO_f')| \leqslant |f(\OO)'|$ we mark all but one point in $ \overleftarrow{f} (\{ \overline{\oo} \})$ for each $\overline{\oo} \in f(\OO_f')$ and redefine $f$ at the marked points to assume distinct values in $f(\OO)'$; this `empties' $\OO_f$ but may not `fill' $f(\OO)'$. The result is an injective map $g \equiv f$ such that $|g (\OO)'| = - {\rm ind} (f)$. In case $|\OO_f'| - |f(\OO_f')| \geqslant |f(\OO)'|$ we mark $|f(\OO)'|$ points of $\OO_f'$ and redefine the value of $f$ at each marked point in turn, mapping the set of marked points to $f(\OO)'$ bijectively. At each stage, $|f(\OO)'|$ is reduced by one; at each stage, if the marked point $\oo$ is not (currently) the next-to-last point of $\overleftarrow{f} (\{\overline{\oo}\})$ for some $\overline{\oo} \in f(\OO_f')$ then $|\OO_f'|$ falls by one but $|f(\OO_f')|$ is unchanged, while if $\oo$ is such a next-to-last point then $|\OO_f'|$ falls by two but $|f(\OO_f')|$ falls by one. This procedure `fills' $f(\OO)'$ but may not `empty' $\OO_f'$; it results in a surjective map $g \equiv f$ such that $|\OO_g'| - |f(\OO_g')| = {\rm ind} (f)$. 
\end{proof}

\medbreak 

\begin{theorem} \label{inper}
If the injective near-surjections $f: \OO \ra \OO$ and $g: \OO \ra \OO$ have equal index then there exist permutations $\lambda \in S_{\OO}$ and $\rho \in S_{\OO}$ such that $\lambda \circ f \equiv g \equiv f \circ \rho$.
\end{theorem} 

\begin{proof} 
($\lambda$) If $\oo \in f(\OO)$ then $\oo = f(\oo')$ for a unique $\oo' \in \OO$ and we define $\lambda (\oo) := g(\oo')$; the resulting map $f(\OO) \ra g(\OO)$ is plainly a bijection. As the finite cardinalities $|f(\OO)'| = - {\rm ind} (f)$ and $|g(\OO)'| = - {\rm ind} (g)$ are equal, there is a bijection $f(\OO)' \ra g(\OO)'$. Piecing together these two bijections manufactures a bijection $\lambda \in S_{\OO}$ such that $\lambda \circ f = g$ in fact. \par 
($\rho$) Let $\oo \in \overleftarrow{g} (f(\OO))$: say $g(\oo) = f(\oo')$ for an $\oo' \in \OO$ that is unique as $f$ is injective; we define $\rho(\oo) := \oo'$ and note that $\rho(\oo)$ actually lies in $\overleftarrow{f} (g(\OO))$. The map $\overleftarrow{g} (f(\OO)) \ra \overleftarrow{f} (g(\OO))$ so defined is plainly a bijection. Notice that 
$$f(\OO)' = (f(\OO)' \cap g(\OO)) \cupdot (f(\OO)' \cap g(\OO)')$$
and 
$$g(\OO)' = (f(\OO) \cap g(\OO)') \cupdot (f(\OO)' \cap g(\OO)')$$
where the left-most sets have equal finite cardinality and the right-most sets are equal, whence 
$$|\overleftarrow{g} (f(\OO))'| = |f(\OO)' \cap g(\OO)| = |f(\OO) \cap g(\OO)'| = | \overleftarrow{f} (g(\OO))'|$$
and so there exists a bijection $\overleftarrow{g} (f(\OO))' \ra \overleftarrow{f} (g(\OO))'$. Piecing together these two bijections manufactures a bijection $\rho \in S_{\OO}$ such that $ f = g \circ \rho$ on $\overleftarrow{g} (f(\OO))$ and therefore $f \equiv g \circ \rho$. 
\end{proof}

\medbreak 

{\bf Remark:} Theorem \ref{inper} and the first part of Theorem \ref{red} together show that if $f: \OO \ra \OO$ and $g: \OO \ra \OO$ are near-bijections with the same non-positive index then  there exist permutations $\lambda \in S_{\OO}$ and $\rho \in S_{\OO}$ such that $\lambda \circ f \equiv g \equiv f \circ \rho$.

\medbreak 

\begin{theorem} \label{surper}
If the surjective near-injections $f: \OO \ra \OO$ and $g: \OO \ra \OO$ have equal index then there exist permutations $\lambda \in S_{\OO}$ and $\rho \in S_{\OO}$ such that $\lambda \circ f \equiv g \equiv f \circ \rho$.
\end{theorem} 

\begin{proof} 
($\lambda$) Let $\oo \in f(\OO_f \cap \OO_g)$: thus, $\oo = f(\oo')$ for an $\oo' \in \OO_f \cap \OO_g$ that is unique by definition of $\OO_f$; we define $\lambda(\oo) := g(\oo')$ and note that $\lambda(\oo)$ actually lies in $g(\OO_f \cap \OO_g)$. By symmetry, the resulting map $\lambda: f(\OO_f \cap \OO_g) \ra g(\OO_f \cap \OO_g)$ is a bijection such that $\lambda \circ f$ and $g$ agree on $\OO_f \cap \OO_g$. As $\OO_f \cap \OO_g$ is cofinite, $\lambda$ may be extended by a bijection $f(\OO_f \cap \OO_g)' \ra g(\OO_f \cap \OO_g)'$ to produce $\lambda \in S_{\OO}$ as required once we have justified the following. \par 
Claim: $f(\OO_f \cap \OO_g)'$ and $g(\OO_f \cap \OO_g)'$ have the same finite cardinality. \par
[We shall simplify appearances by writing $A$ for $\OO_f$ and $B$ for $\OO_g$. Consider the disjoint decomposition 
$$f(A \cap B)' = (\OO \setminus f(A)) \cupdot (f(A) \setminus f(A \cap B)).$$ 
The first term on the right is $f(\OO \setminus A) = f(A')$ according to Theorem \ref{comp}. Regarding the second set, we claim that 
$$f(A) \setminus f(A \cap B) = f(A \setminus (A \cap B)).$$
On the one hand, if $\oo \in f(A) \setminus f(A \cap B)$ then $\oo = f(\oo')$ for some $\oo' \in A$ and this element $\oo'$ cannot lie in $A \cap B$. On the other hand, let $\oo \in A \setminus (A \cap B)$: were $f(\oo)$ to lie in $f(A \cap B)$ it would follow that $f(\oo) = f(\oo')$ for some $\oo' \in A \cap B$ whence the injectivity of $f|_A$ would force $\oo = \oo' \in A \cap B$ (contradiction); this places $f(\oo)$ in $f(A) \setminus f(A \cap B)$. The injectivity of $f|_A$ makes its restriction 
$$A \setminus (A \cap B) \ra f(A) \setminus f(A \cap B)$$
a bijection, while of course $A \setminus (A \cap B) = A \cap B'$ is finite. The foregoing disjoint decomposition of $f(A \cap B)'$ therefore yields 
$$|f(A \cap B)'| = |f(A)'| + |A \cap B'|$$
with a parallel expression for $|g(A \cap B)'|$. Now 
$$|f(A \cap B)'| + |A' \cap B'| = |f(A')| + |A \cap B'| + |A' \cap B'| = |f(A')| + |B'|$$
and similarly 
$$|g(A \cap B)'| + |A' \cap B'| = |g(B')| + |A' \cap B| + |A' \cap B'| = |g(B')| + |A'|.$$
Here, ${\rm ind} (f) = {\rm ind} (g)$ may be written $|f(A')| + |B'| = |g(B')| + |A'|$; consequently, cancellation of $|A' \cap B'|$ justifies the claim.] \par 
($\rho$) We continue to write $A = \OO_f$ and $B = \OO_g$. Let $\oo \in B \cap \overleftarrow{g} (f(A))$: thus $\oo \in B$ and $g(\oo) \in f(A)$; say $g(\oo) = f(\oo')$ for a unique $\oo' \in A$. We define $\rho(\oo) := \oo'$  and note that $\rho(\oo) \in A \cap \overleftarrow{f} (g(B))$. By symmetry, the resulting map $\rho: B \cap \overleftarrow{g} (f(A)) \ra A \cap \overleftarrow{f} (g(B))$ is a bijection such that $f \circ \rho = g$.  This bijection may be extended by a bijection $ B \cap \overleftarrow{g} (f(A))' \ra A \cap \overleftarrow{f} (g(B))'$ to produce $\rho \in S_{\OO}$ as required once we have justified the following. \par 
Claim: $B \cap \overleftarrow{g} (f(A))'$ and $A \cap \overleftarrow{f} (g(B))'$ have the same finite cardinality. \par 
[Consider the disjoint decomposition
$$B \cap \overleftarrow{g} (f(A))' = (\OO \setminus B) \cupdot (B \setminus \overleftarrow{g} (f(A))).$$
Here, $g$ restricts to a bijection from $B \setminus \overleftarrow{g} (f(A))) = B \cap \overleftarrow{g} (f(A))' = B \cap \overleftarrow{g} (f(A'))$ (the last equality holding by virtue of Theorem \ref{comp}) to $g(B) \cap f(A')$: certainly $g$ maps $B \cap \overleftarrow{g} (f(A')) \subseteq B$ injectively to $g(B) \cap f(A')$; if $\oo \in g(B) \cap f(A')$ then there exist $\overline{\oo} \in B$ and $\oo' \in A'$ so that $\oo = g(\overline{\oo}) = f(\oo')$ whence $\oo = g(\overline{\oo})$ with $\overline{\oo} \in B \cap \overleftarrow{g} (f(A'))$. It follows that 
$$|B \setminus \overleftarrow{g} (f(A))| = |g(B) \cap f(A')| = |g(B) \cap f(A)'|$$
whence the foregoing disjoint decomposition of $B \cap \overleftarrow{g} (f(A))'$ yields 
$$|(B \cap \overleftarrow{g} (f(A)))'| = |B'| + |g(B) \cap f(A)'|$$
and similarly 
$$|(A \cap \overleftarrow{f} (g(B)))'| = |A'| + |f(A) \cap g(B)'|.$$
All that remains is elementary arithmetic on these two equations: add $|f(A)' \cap g(B)'|$ throughout; invoke equality of indices in the form $|B'| + |f(A')| = |A'| + |g(B')|$; and finally cancel $|f(A)' \cap g(B)'|$ throughout.]
\end{proof} 

\medbreak 

{\bf Remark:} Theorem \ref{inper} and the second part of Theorem \ref{red} together show that if $f: \OO \ra \OO$ and $g: \OO \ra \OO$ are near-bijections with the same non-negative index then  there exist permutations $\lambda \in S_{\OO}$ and $\rho \in S_{\OO}$ such that $\lambda \circ f \equiv g \equiv f \circ \rho$. 

\medbreak 

We close this section by remarking on a difference between injective near-surjections and surjective near-injections; this stems from the fact that when $f$ is a near-bijection, $f(\OO)'$ is essentially featureless whereas $\OO_f'$ has internal structure. As we saw in the proof of Theorem \ref{inper}, if $f$ and $g$ are injective near-surjections with the same index then there exists a permutation $\lambda \in S_{\OO}$ such that $\lambda \circ f = g$ (true equality). By contrast, if $f$ and $g$ are surjective near-injections then $g$ need not be obtained from $f$ by composition with a permutation on either side; the best that we can hope for in general is almost equality. The following result shows what happens when we demand equality. 

\medbreak 

\begin{theorem} 
Let $f: \OO \ra \OO$ and $g: \OO \ra \OO$ be surjective near-injections. If 
$$f(\OO_f') = g(\OO_g') =: \overline{\OO}$$ 
and 
$$(\forall \overline{\oo} \in \overline{\OO}) \; \; \; \; \; |\overleftarrow{f} (\{\overline{\oo}\})| =  |\overleftarrow{g} (\{\overline{\oo}\})| $$
then there exists a permutation $\rho \in S_{\OO}$ such that $g = f \circ \rho.$
\end{theorem} 

\begin{proof} 
Note by Theorem \ref{comp} that 
$$f(\OO_f) = f(\OO_f')' = g(\OO_g')' = g(\OO_g).$$
Let $\oo \in \OO_g$: then $g(\oo) \in g(\OO_g) = f(\OO_f)$ so that $g(\oo) = f(\oo')$ for some $\oo' \in \OO_f$ that is unique by definition of $\OO_f$; define $\rho (\oo) := \oo'$. By symmetry, this defines a bijection $\rho: \OO_g \ra \OO_f$ such that $f \circ \rho = g|_{\OO_g}$. We construct a bijection $\rho: \OO_g' \ra \OO_f'$ as follows: for each $\overline{\oo} \in \overline{\OO}$ we may choose a bijection $\overleftarrow{g} (\{ \overline{\oo}\} ) \ra \overleftarrow{f} (\{\overline{\oo}\})$; these individual bijections piece together to give a bijection 
$$\OO_g' =  \bigcupdot_{\overline{\oo} \in \overline{\OO}} \overleftarrow{g} (\{ \overline{\oo} \}) \xrightarrow{\rho}   \bigcupdot_{\overline{\oo} \in \overline{\OO}} \overleftarrow{f} (\{ \overline{\oo} \}) = \OO_f'.$$
Finally, the bijections $\OO_g \ra \OO_f$ and $\OO_g' \ra \OO_f'$ collate to provide a permutation $\rho \in S_{\OO}$ with the desired property $g = f \circ \rho$. 
\end{proof} 

\medbreak 

It is readily verified that this sufficient condition for the existence of $\rho$ is also necessary. 

\medbreak 

\section{$\mathbb{G}_{\OO}$}

\medbreak 

In this section, we frame our results regarding near-bijections and their indices in a group-theoretical setting: as we shall see, the $\equiv$-classes of near bijections $\OO \ra \OO$ constitute a group, on which the index descends to define a $\mathbb{Z}$-valued group homomorphism whose kernel comprises the $\equiv$-classes of (almost) bijections. 

\medbreak 

Denote by $\mathcal{I}_{\OO}$ the set comprising all near-injections from $\OO$ to itself. As a consequence of Theorem \ref{incom}, $\mathcal{I}_{\OO}$ is closed under the (associative) operation of composition; moreover, the identity map ${\rm I}$ on $\OO$ serves as an identity element. Consequently, $\mathcal{I}_{\OO}$ is a monoid. According to Theorem \ref{com}, composition in $\mathcal{I}_{\OO}$ respects almost equality. We denote the corresponding quotient monoid by  
$$\mathbb{I}_{\OO} := \mathcal{I}_{\OO}/\equiv.$$

\medbreak 

Denote by $\mathcal{G}_{\OO}$ the set comprising all near-bijections from $\OO$ to itself. Theorem \ref{bicom} informs us that $\mathcal{G}_{\OO}$ is closed under composition; thus $\mathcal{G}_{\OO}$ also is a monoid with ${\rm I}$ as identity element. Theorem \ref{com} ensures that composition in $\mathcal{G}_{\OO}$ respects almost equality and we obtain a quotient monoid 
$$\mathbb{G}_{\OO} := \mathcal{G}_{\OO}/\equiv.$$

\medbreak 

\begin{theorem} 
The quotient monoid $\mathbb{G}_{\OO}$ is a group. 
\end{theorem} 

\begin{proof} 
Let $f \in \mathcal{G}_{\OO}$ be a near-bijection: say $f$ restricts to a bijection from the cofinite set $A \subseteq \OO$ to the cofinite set $B \subseteq \OO$; we must fashion a near-bijection $g \in \mathcal{G}_{\OO}$ whose $\equiv$-class $[g]$ is inverse to $[f]$. Simply define $g$ to be the inverse of $A \xrightarrow{f} B$ on $B$ and extend over the complement $B'$ arbitrarily. By construction, $g \in \mathcal{G}_{\OO}$ satisfies both $g \circ f|_A = {\rm I}|_A$ and $f \circ g|_B = {\rm I}|_B$; now $D(g \circ f, {\rm I}) \subseteq A'$ and $D(f \circ g, {\rm I}) \subseteq B'$ are finite, thus $g \circ f \equiv {\rm I} \equiv f \circ g$ and so $[g][f] = {\rm I} = [g][f]$ as required. 
\end{proof} 

\medbreak 

{\bf Remark:} Let $f, g \in \mathcal{I}_{\OO}$ be near-injections such that $g \circ f \equiv {\rm I} \equiv f \circ g$. To say that $g \circ f \equiv {\rm I}$ is to say that $g \circ f|_A = {\rm I}|_A$ for some cofinite $A \subseteq \OO$; it follows that $g(\OO) \supseteq A$ is cofinite, whence $g$ is also a near-surjection and hence a near-bijection by Theorem \ref{near}; likewise, $f \circ g \equiv {\rm I}$ implies that $f$ is a near-bijection. In this way, we identify $\mathbb{G}_{\OO}$ as the group of units in $\mathbb{I}_{\OO}.$ 

\medbreak 

We mentioned in Theorem \ref{almostnear} that each almost bijective map is a near-bijection. Denote by $\mathbb{S}_{\OO}$ the set of $\equiv$-classes of almost bijective maps $\OO \ra \OO$; plainly, $\mathbb{S}_{\OO}$ is a subgroup of $\mathbb{G}_{\OO}$. This subgroup $\mathbb{S}_{\OO}$ arises in another way, as follows. The circumstance that each permutation of $\OO$ is certainly a near-bijection of $\OO$ is expressed by the inclusion map $S_{\OO} \ra \mathcal{G}_{\OO}$, which is of course a homomorphism (of monoids); the quotient map $\mathcal{G}_{\OO} \ra \mathbb{G}_{\OO}$ is also a (monoid) homomorphism. The composite map $S_{\OO} \ra \mathbb{G}_{\OO}$ is then a group homomorphism, with image exactly $\mathbb{S}_{\OO}$. Notice that the kernel of the resulting surjective homomorphism $S_{\OO} \ra \mathbb{S}_{\OO}$ comprises precisely all the permutations of $\OO$ that act as the identity on a cofinite set; these permutations make up the {\it finitary} symmetric group $FS_{\OO}$. The group $\mathbb{S}_{\OO}$ is therefore canonically isomorphic to the quotient $S_{\OO} / FS_{\OO}.$ 

\medbreak 

According to Theorem \ref{insensitive}, the index is constant on $\equiv$-classes and therefore descends to a map 
$${\rm Ind} : \mathbb{G}_{\OO} \ra \mathbb{Z}$$
given by the requirement that if $f \in \mathcal{G}_{\OO}$ then 
$${\rm Ind} [f] = {\rm ind} (f).$$

\medbreak 

 Theorem \ref{coset} has the consequence that ${\rm Ind}$ is constant on each coset of $\mathbb{S}_{\OO}$, whether left coset or right coset. In the opposite direction, the Remark after Theorem \ref{inper} implies that if ${\rm Ind}$ has the same non-positive value on two elements of $\mathbb{G}_{\OO}$ then these elements lie in the same left coset and the same right coset of $\mathbb{S}_{\OO}$, while the Remark after Theorem \ref{surper} implies the same conclusion for matching non-negative values of ${\rm Ind}$. This at once establishes that the subgroup $\mathbb{S}_{\OO}$ of $\mathbb{G}_{\OO}$ is normal. 

\medbreak 

We claim that ${\rm Ind}$ is actually a group homomorphism. To see this, we take a minor detour and briefly examine some especially basic near-bijections. 

\medbreak 

Let $u: \OO \ra \OO$ be injective with $|u(\OO)'| = 1$: say $u(\OO)' = \{ \oo_0 \}$; of course, $u \in \mathcal{G}_{\OO}$ with ${\rm ind} (u) = -1$. Define $v : \OO \ra \OO$ as follows: $v ( \oo_0) = \oo_0$;  if $\oo = u(\oo') \in u(\OO)$ then $v(\oo) = \oo'$. Notice that $v(\OO) = \OO$ while $\OO_v' = \{ \oo_0, u(\oo_0) \}$ and $v(\OO_v') = \{ \oo_0 \}$; so $v \in \mathcal{G}_{\OO}$ with ${\rm ind} (v) = 1$. Plainly, $v \circ u = {\rm I}$ while $u \circ v|_{u(\OO)} = {\rm I}|_{u(\OO)}$. Thus the elements $[u]$ and $[v]$ of $\mathbb{G}_{\OO}$ are mutual inverses. 

\medbreak 

\begin{theorem} \label{n}
If $n$ is a positive integer then 
$$u^n (\OO)' = \{ \oo_0, u(\oo_0), \ \dots , u^{n - 1} (\oo_0) \}$$
and 
$$\OO_{v^n}' = \{ \oo_0, , u(\oo_0), \ \dots , u^n (\oo_0) \}.$$
\end{theorem} 

\begin{proof} 
In each case, the proof is a routine but  instructive exercise on induction. 
\end{proof} 

\medbreak 

{\bf Remark:} It follows at once that ${\ind} (u^n) = -n$ and (because $v^n$ has constant value $\oo_0$ on $\OO_{v^n}'$) that ${\rm ind} (v^n) = n$. 

\medbreak 

Now $[u]$ and $[v]$ are the two generators of an infinite cyclic subgroup of $\mathbb{G}_{\OO}$; of course, if $m, n \in \mathbb{Z}$ then $[v]^m [v]^n = [v]^{m + n}$. 

\medbreak 

This detour behind us, we are ready. 

\begin{theorem} 
The index map ${\rm Ind} : \mathbb{G}_{\OO} \ra \mathbb{Z}$ is a group homomorphism. 
\end{theorem} 

\begin{proof} 
Let $f \in \mathcal{G}_{\OO}$ have index $m$ and $g \in \mathcal{G}_{\OO}$ have index $n$. By the Remarks after Theorem \ref{inper} and Theorem \ref{surper},  there exist permutations $\mu \in S_{\OO}$ and $\nu \in S_{\OO}$ such that $f \equiv \mu \circ v^m$ and $g \equiv v^n \circ \nu$. Now 
$$[f] [g] = [\mu] [v]^m [v]^n [\nu] = [\mu] [v]^{m + n} [\nu]$$
whence by Theorem \ref{coset} we conclude that 
$${\rm Ind} [f][g] = m + n = {\rm Ind} [f] + {\rm Ind} [g].$$

\end{proof} 

\medbreak 

Observe that $\mathbb{S}_{\OO}$ arises here in yet another way: Theorem \ref{indzero} reveals that it is precisely the kernel of ${\rm Ind}$. 

\medbreak 

To put matters in other words, we have produced a short exact sequence of groups 
$${\rm I} \ra \mathbb{S}_{\OO} \ra \mathbb{G}_{\OO} \ra \mathbb{Z} \ra 0.$$
As $\mathbb{Z}$ is infinite cyclic, this short exact sequence splits: a splitting homomorphism is given by 
$$\mathbb{Z} \ra \mathbb{G}_{\OO} : n \mapsto [v]^n.$$

\medbreak 

We can say a little more. Let $\phi: \mathbb{Z} \ra \mathbb{G}_{\OO}$ be any homomorphism that splits this sequence in the sense that ${\rm Ind} \circ \phi$ is the identity map on $\mathbb{Z}$. The first part of Theorem \ref{red} places in the $\equiv$-class $\phi (-1) \in \mathbb{G}_{\OO}$ an injective self-map of $\OO$ whose range contains all but one point of $\OO$; by relabelling, we may take $\phi(-1) = [u]$ with $u: \OO \ra \OO$ the self-map introduced prior to Theorem \ref{n}. For each natural number $n$ let us write $\oo_n = u^n (\oo_0)$ and put $\OO_0 = \{ \oo_n : n \geqslant 0 \}$. Define a permutation $\pi \in S_{\OO}$ as follows: if $m \geqslant 0$ then $\pi (\oo_{2m}) = \oo_{2 m + 1}$ and $\pi (\oo_{2 m + 1}) = \oo_{2 m}$; while $\pi$ fixes the complement $\OO_0'$ pointwise. By direct calculation, $\pi \circ u$ fixes the points of $\OO_0$ with even labels and shifts those with odd labels, whereas $u \circ \pi$ fixes the points of $\OO_0$ with odd labels and shifts those with even labels; both composites agree with $u$ on $\OO_0'$. This shows that the disagreement set $D(\pi \circ u, u \circ \pi)$ is precisely the countably infinite set $\OO_0$: thus $\pi \circ u$ and $u \circ \pi$ are not almost equal, so $[u]$ and $[\pi]$ do not commute. In short, the image of a splitting homomorphism cannot lie in the centre of $\mathbb{G}_{\OO}$; though a semidirect product, $\mathbb{G}_{\OO}$ is not direct. 

\medbreak 

We leave to the reader the pleasure of tracing the parallels between the Fredholm theory for operators and the theory developed in this paper.

\bigbreak

\begin{center} 
{\small R}{\footnotesize EFERENCES}
\end{center} 
\medbreak 

Chapter XI of [1] addresses the Fredholm theory for operators and [2] covers the theory of permutation groups, both finite and infinite; [3] includes accounts of permutation groups and short exact sequences, both split and otherwise. 

\medbreak 

[1] J.B. Conway, {\it A Course in Functional Analysis}, Springer GTM 96 (1985). 

\medbreak 

[2] J.D. Dixon and B. Mortimer, {\it Permutation Groups}, Springer GTM 163 (1996). 

\medbreak 

[3] W.R. Scott, {\it Group Theory}, Prentice-Hall (1964); Dover (1987).

\medbreak

\end{document}